\documentclass[11pt, reqno]{amsart}
\usepackage{amssymb,amsthm, amsmath, amsfonts}
\usepackage{graphics}
\usepackage{hyperref}
\usepackage[all]{xy}
\usepackage{enumerate}
\usepackage{mathrsfs}
\usepackage{hyperref}

\theoremstyle{plain}
\newtheorem{theorem}{Theorem}
\newtheorem{lemma}[theorem]{Lemma}
\newtheorem{corollary}[theorem]{Corollary}
\newtheorem*{claim}{Claim}

\theoremstyle{definition}
\newtheorem{remark}[theorem]{Remark}
\newtheorem{notation}[theorem]{Notation}

\newcommand{\Stab}{{\mathrm{Stab}}}
\newcommand{\Tot}{{\mathrm{Tot}}}
\newcommand{\Z}{{\mathbb{Z}}}

\title{Complex of injective words revisited}

\author{Wee Liang Gan}
\address{Department of Mathematics, University of California, Riverside, CA 92521, USA.}
\email{wlgan@math.ucr.edu}

\subjclass[2010]{18G35}
\keywords{injective word, homological stability, wreath-product group}

\begin{document}

\begin{abstract}
We give a simple proof that (a generalization of) the complex of injective words has vanishing homology in all except the top degree.
\end{abstract}

\maketitle

\section{Introduction}

Let $A$ be a finite set. An \emph{injective word} of length $r\geqslant 0$ is a sequence $(a_1,\ldots, a_r)$ of pairwise distinct elements of $A$. Let $K(A)$ be the semi-simplicial set whose $(r-1)$-simplices are the injective words of length $r$, for every $r\geqslant 1$. The face maps of $K(A)$ are defined by 
\begin{equation*}
d_{i-1} (a_1,\ldots, a_r) := (a_1, \ldots, \widehat{a_i}, \ldots, a_r) \quad\mbox{ for } i=1,\ldots, r,
\end{equation*}
where $\widehat{x}$ means that the entry $x$ is omitted. In other words, $K(A)$ is the semi-simplicial set of ordered simplices of the abstract simplex whose set of vertices is $A$. We write $|K(A)|$ for the geometric realization of $K(A)$. For example, if $A=\{a,b\}$, then $|K(A)|$ is homeomorphic to a circle:
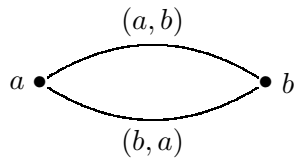
\begin{equation*}
\begin{xy}
(0,0)*{\bullet}="A"-(3,0)*{a}; 
(30,0)*{\bullet}="B"+(3,0)*{b}; 
"A"; "B" **\crv{(15,10)}\POS?*+!D{(a,b)};
"A"; "B" **\crv{(15,-10)}\POS?*+!U{(b,a)}
\end{xy}
\end{equation*}

In a 1979 paper, Farmer \cite[Theorem 5]{Fa} shows that the reduced homology of $K(A)$ vanishes in all degrees $\neq |A|-1$. Subsequently, a new proof was found by Bj\"orner and Wachs \cite[Theorem 6.1]{BW} using their theory of CL-shellable posets. A simpler proof of Farmer's result was given by Kerz \cite[Theorem 1]{Ke} in 2004. Both the proofs of Farmer and Kerz proceed by somewhat ad hoc calculations. More recently, topological proofs of Farmer's result are given by Bestvina \cite[Claim in the proof of Proposition 6]{Be} and Randal-Williams \cite[Proposition 3.2]{Ra}. We should mention that Bj\"orner-Wachs and Randal-Williams actually proved the stronger result that $|K(A)|$ is homotopy equivalent to a wedge of spheres of dimension $|A|-1$.

The purpose of our present note is to give a simple and natural algebraic proof of Farmer's result; indeed, our proof is a straightforward exercise on the spectral sequence of a filtered complex. Our interest in this result stems from the crucial role it plays in Quillen's method \cite{Qu} for proving homological stability of the symmetric groups. Following Hatcher and Wahl \cite{HW}, we shall formulate and prove a slightly more general theorem so that it can be applied in the proof of homological stability of wreath-product groups.

\begin{notation} \label{wreath product notation}
We write $S_n$ for the symmetric group on $\{1,\ldots, n\}$. For any group $G$, we write $G_n$ for the wreath-product group $G\wr S_n$, that is, $G_n := S_n\ltimes G^n$. In particular, $G_0$ is the trivial group.
\end{notation}

\section{The main result}

Recall that $A$ denotes a finite set. Let $\Gamma$ be a nonempty set. We define a chain complex $C_*(A)$ concentrated in degrees $0,\ldots, |A|$ as follows. Let $C_r(A)$ be the free abelian group generated by the set $\Delta_r(A)$ consisting of all elements of the form $(a_1, \ldots, a_r, \gamma_1, \ldots, \gamma_r)$ where $a_1,\ldots, a_r$ are pairwise distinct elements of $A$, and $\gamma_1,\ldots, \gamma_r$ are any elements of $\Gamma$; in particular, $C_r(A)=\Z$ if $r=0$. The differential is defined by 
\begin{equation*}
d(a_1, \ldots, a_r, \gamma_1, \ldots, \gamma_r) := \sum_{i=1}^r (-1)^{i-1} (a_1, \ldots, \widehat{a_i}, \ldots, a_r, \gamma_1, \ldots, \widehat{\gamma_i}, \ldots, \gamma_r).
\end{equation*} 
In particular, $d(a_1,\gamma_1)=1$. 

For any chain complex $C_*$ and positive integer $p$, we shall write $C_{*-p}$ for the $p$-fold suspension of $C_*$.

\begin{remark}
If $\Gamma$ is a singleton set, then $C_*(A)$ is the augmented chain complex of $K(A)$ with degrees shifted up by 1. (Topologically, it is more natural to place a word of length $r$ in degree $r-1$. Algebraically, it seems more natural for us to place a word of length $r$ in degree $r$.) 

If $\Gamma$ is a group $G$, then $C_*(A)$ is the augmented chain complex of the semi-simplicial set $W_{|A|} (\emptyset, \{1\})$ (defined by Randal-Williams and Wahl in \cite[Definition 2.1]{RaW}) with degrees shifted up by 1, associated to the category $\mathrm{FI}_G$ (defined by Sam and Snowden in \cite{SS}).
\end{remark}

\begin{theorem} \label{connectivity}
If $r<|A|$, then $H_r(C_*(A))=0$.
\end{theorem}

\begin{proof}
Set $n=|A|$. We use induction on $n$. The base case $n=0$ is trivial. 

Suppose $n>0$. Choose and fix an element $a\in A$. For each $r\geqslant 0$, there is an increasing filtration on $C_r(A)$ defined by letting $F_pC_r(A)$ be the subgroup spanned by all elements $(a_1, \ldots, a_r, \gamma_1, \ldots, \gamma_r)$ such that none of $a_{p+1},\ldots, a_r$ is equal to $a$. This gives an increasing filtration on the complex $C_*(A)$ and hence a first-quadrant spectral sequence:
\begin{equation*}
E^1_{p,q} = H_{p+q}\left( F_pC_*(A)/F_{p-1}C_*(A) \right) \Rightarrow H_{p+q}(C_*(A)).
\end{equation*}

Observe that $F_0C_*(A)=C_*(A\setminus\{a\})$. For each $p\geqslant 1$, there is an isomorphism of chain complexes:
\begin{equation*}
F_pC_*(A)/F_{p-1}C_*(A) \cong \bigoplus_{(a_1,\ldots,a_{p-1}, a, \gamma_1,\ldots,\gamma_p)\in \Delta_p(A)} C_{*-p}(A\setminus\{a_1, \ldots, a_{p-1}, a\})
\end{equation*}
where an element on the left hand side represented by $(a_1, \ldots, a_r, \gamma_1, \ldots, \gamma_r)\in F_pC_r(A)$ with $a_p=a$ is identified, on the right hand side, with the element $(a_{p+1}, \ldots, a_r, \gamma_{p+1}, \ldots, \gamma_r)$ in the direct summand $C_{r-p}(A\setminus\{a_1, \ldots, a_{p-1}, a\})$ indexed by $(a_1,\ldots,a_p,\gamma_1,\ldots,\gamma_p)\in \Delta_p(A)$. 

By the induction hypothesis, one has: 
\begin{gather*}
E^1_{0,q}=0 \quad \mbox{ whenever} \quad q<n-1;\\
E^1_{p,q}=0 \quad \mbox{ whenever } \quad p \geqslant 1 \mbox{ and } p+q<n.
\end{gather*}
Therefore, it only remains to show that $E^\infty_{0,n-1}=0$.

We have:
\begin{equation*}
E^1_{0,n-1} = H_{n-1}(C_*(A\setminus\{a\})), \qquad E^1_{1,n-1} = \bigoplus_{\gamma_1\in \Gamma} H_{n-1}(C_*(A\setminus\{a\})).
\end{equation*}
The restriction of the differential $d^1 : E^1_{1,n-1} \to E^1_{0,n-1}$ to each direct summand in the above decomposition of $E^1_{1,n-1}$ is the identity map on $H_{n-1}(C_*(A\setminus\{a\}))$; this follows from the identity
\begin{multline*}
d(a, a_2, \ldots, a_n, \gamma_1, \ldots, \gamma_n) = (a_2, \ldots, a_n, \gamma_2, \ldots, \gamma_n) \\- \sum_{i=1}^{n-1} (-1)^{i-1} (a, a_2, \ldots, \widehat{a_{i+1}},\ldots, \gamma_1, \gamma_2, \ldots, \widehat{\gamma_{i+1}}, \ldots) .
\end{multline*}
In particular, the map $d^1 : E^1_{1,n-1} \to E^1_{0,n-1}$ is surjective. It follows that $E^2_{0,n-1}=0$, and we are done.
\end{proof}

\begin{remark}
It was pointed out to us by the referee that the last step in the above proof can be replaced by the following argument. Fixing an element $e \in \Gamma$, the map 
\begin{equation*}
(a_1,\ldots,a_r,\gamma_1,\ldots,\gamma_r) \mapsto (a,a_1,\ldots,a_r,e,\gamma_1,\ldots,\gamma_r)
\end{equation*}
gives a null-homotopy for the inclusion map $C_*(A-\{a\}) \to C_*(A)$. Thus, the edge map $E^1_{0,n-1} \longrightarrow E^\infty_{0,n-1} \subset H_{n-1}(C_*(A))$ is zero; hence, $E^\infty_{0,n-1}=0$.
\end{remark}

\begin{remark}
Let $n=|A|$. Since $H_n(C_*(A))$ is a subgroup of the free abelian group $C_n(A)$, it is a free abelian group; we make the following observations on its rank.

(i) Suppose $|\Gamma|=\infty$. If $n =|A|\geqslant 1$, then $H_n(C_*(A))$ is a free abelian group of infinite rank. This is clear if $n=1$. For $n>1$, it follows from noticing that by induction, the kernel of $d^1 : E^1_{1,n-1} \to E^1_{0,n-1}$ (in the spectral sequence in the proof above) is a free abelian group of infinite rank, and which is $E^\infty_{1,n-1}$.

(ii) Suppose $|\Gamma|=\ell < \infty$. If $n=|A|\geqslant 1$, then it follows from Theorem \ref{connectivity} and the Euler-Poincar\'e principle that $H_n(C_*(A))$ is a free abelian group of rank $d(\ell)_n$, where 
\begin{equation*}
d(\ell)_n := \sum_{i=0}^n \frac{(-1)^i n!\ell^{n-i}}{i!}.
\end{equation*}
For $\ell=1$, this observation is due to Farmer \cite[Remark on page 613]{Fa} and Reiner-Webb \cite[Proposition 2.1]{RW}. It is well known that $d(1)_n$ is equal to the number of derangements in the symmetric group $S_n$. Let us give a similar interpretation of $d(\ell)_n$ for any $\ell\geqslant 1$. Fixing an element $e\in \Gamma$, we claim that $d(\ell)_n$ is the number of elements $(\pi, \gamma_1, \ldots, \gamma_n) \in S_n \times \Gamma^n$ such that: if $1\leqslant a\leqslant n$ and $\pi(a)=a$, then $\gamma_a \neq e$. To see this, let 
\begin{equation*}
T_a := \{ (\pi, \gamma_1, \ldots, \gamma_n) \in S_n \times \Gamma^n \mid \pi(a)=a \mbox{ and } \gamma_a=e \} \quad \mbox{ for }a = 1,\ldots, n.
\end{equation*}
Then, for any $1\leqslant a_1 < \cdots < a_i \leqslant n$, one has $|T_{a_1} \cap \cdots \cap T_{a_i}| = (n-i)!\ell^{n-i}$. Hence, by the inclusion-exlusion principle, we have $d(\ell)_n = |S_n \times \Gamma^n| - |T_1 \cup \cdots \cup T_n|$, as claimed. 

We note that for any group $G$, the wreath-product group $G_n$ (see Notation \ref{wreath product notation}) acts on $\{1,\ldots, n\}\times G$ by $(\pi; g_1, \ldots g_n) \cdot (a, \gamma) := (\pi(a), g_a \gamma)$, where $(\pi; g_1,\ldots,g_n)\in G_n$ and $(a,\gamma) \in \{1,\ldots, n\} \times G$. When $G$ is a finite group of order $\ell$, the integer $d(\ell)_n$ is equal to the number of elements of $G_n$ which has no fixed point in $\{1,\ldots, n\}\times G$.
\end{remark}

\section{Application to homological stability}

Nakaoka \cite[Corollary 6.7]{Na} proved that the natural inclusion map $S_{n-1} \to S_n$ induces an isomorphism in homology $H_m(S_{n-1}) \to H_m(S_n)$ if $n>2m$. His result was generalized by Hatcher and Wahl \cite[Proposition 1.6]{HW} to wreath-product groups (although as they noted in their paper the generalization might have been known for a long time).

\begin{corollary} \label{stability}
Let $G$ be a group. The natural inclusion map $G\wr S_{n-1} \to G\wr S_n$ induces an isomorphism in homology $H_m(G\wr S_{n-1}) \to H_m(G\wr S_n)$ if $n>2m$.
\end{corollary}

Corollary \ref{stability} follows from Theorem \ref{connectivity} by a standard argument of Quillen; see \cite[Section 5]{HW}. We give the details of this argument below since our injectivity range $n>2m$ is better than the one stated in \cite[Proposition 1.6]{HW} by 1.

\emph{From now on, we set $A=\{1,\ldots, n\}$ and $\Gamma=G$, where $n$ is an integer $\geqslant 1$ and $G$ is a group.} 

There is a natural action of $G_n$ (see Notation \ref{wreath product notation}) on $C_r(A)$ defined by
\begin{equation*}
(\pi; g_1,\ldots,g_n)\cdot (a_1,\ldots,a_r, \gamma_1,\ldots, \gamma_r) := (\pi(a_1), \ldots, \pi(a_r), g_{a_1}\gamma_1,\ldots, g_{a_r}\gamma_r ),
\end{equation*}
where $(\pi; g_1,\ldots, g_n)\in G_n$ and $(a_1,\ldots,a_r, \gamma_1,\ldots, \gamma_r)\in C_r(A)$. Define the map \begin{equation*}
d_{i-1} : C_r(A) \longrightarrow C_{r-1}(A) \quad\mbox{ for } i=1,\ldots, r,
\end{equation*}
by
\begin{equation*}
d_{i-1} (a_1,\ldots, a_r, \gamma_1,\ldots, \gamma_r) := (a_1, \ldots, \widehat{a_i}, \ldots, a_r, \gamma_1, \ldots, \widehat{\gamma_i},\ldots, \gamma_r).
\end{equation*}
Since the map $d_{i-1}$ is $G_n$-equivariant, there is an induced map 
\begin{equation*}
(d_{i-1})_* : H_*(G_n; C_r(A)) \to H_*(G_n; C_{r-1}(A)).
\end{equation*}

The group $G_n$ acts transitively on the basis $\Delta_r(A)$ of $C_r(A)$. For any $x\in \Delta_r(A)$, we write $\Stab(x)$ for its stablizer in $G_n$. By Shapiro's lemma, the natural inclusion map $\Stab(x) \to G_n$ and the map $\Z \to C_r(A),\; \lambda \mapsto \lambda x$ induce an isomorphism
\begin{equation*}
\alpha(x)_* : H_*(\Stab(x)) \longrightarrow H_*(G_n; C_r(A)).
\end{equation*} 

Denote by $e\in G$ the identity element. Let 
\begin{equation*}
x_r := (n-r+1,\ldots, n, e, \ldots, e) \in C_r(A).
\end{equation*}
Then $\Stab(x_r) = G_{n-r} \leqslant G_n$. In particular, $x_0 = 1 \in \Z$ and $\Stab(x_0) = G_n$. We write $\iota: \Stab(x_r) \to \Stab(x_{r-1})$ for the natural inclusion map.

\begin{lemma} \label{lemma for quillen argument}
For every $i=1,\ldots, r$, the following diagram commutes:
\begin{equation*}
\xymatrix{
H_*(\Stab(x_r)) \ar[rr]^{\iota_*} \ar[d]_{\alpha(x_r)_*} && H_*(\Stab(x_{r-1})) \ar[d]^{\alpha(x_{r-1})_*} \\	
H_*(G_n; C_r(A)) \ar[rr]^-{(d_{i-1})_*} && H_*(G_n; C_{r-1}(A))
}.
\end{equation*}
\end{lemma}
\begin{proof}
The diagram clearly commutes for $i=1$ because $d_0(x_r)=x_{r-1}$.

Suppose $i>1$. Let $y := d_{i-1}(x_r)$, so
\begin{equation*}
y = (n-r+1, \ldots, \widehat{n-r+i}, \ldots, n, e, \ldots, e) \in C_{r-1}(A).
\end{equation*}
Write $\jmath : \Stab(x_r) \to \Stab(y)$ for the natural inclusion map. Then there is a commuting diagram:
\begin{equation} \label{commuting diagram 1}
\xymatrix{
H_*(\Stab(x_r)) \ar[rr]^{\jmath_*} \ar[d]_{\alpha(x_r)_*} && H_*(\Stab(y)) \ar[d]^{\alpha(y)_*} \\	
H_*(G_n; C_r(A)) \ar[rr]^-{(d_{i-1})_*} && H_*(G_n; C_{r-1}(A))
}.
\end{equation}

Let $\mu$ be the cyclic permutation $(n-r+1, \ldots, n-r+i)\in S_n$ and let 
\begin{equation*}
t:=(\mu; e,\ldots, e)\in G_n.
\end{equation*} 
Then $t$ has the properties that $t\cdot y = x_{r-1}$ and 
\begin{equation} \label{commute with stabilizer}
tut^{-1} = u \quad \mbox{ for each } u\in \Stab(x_r).
\end{equation}
We have a commuting diagram
\begin{equation} \label{commuting diagram 2}
\xymatrix{
H_*(\Stab(y)) \ar[rr]^{\kappa_*} \ar[d]_{\alpha(y)_*} && H_*(\Stab(x_{r-1})) \ar[d]^{\alpha(x_{r-1})_*} \\	
H_*(G_n; C_{r-1}(A)) \ar[rr]^-{\delta_*} && H_*(G_n; C_{r-1}(A))
}
\end{equation}
where the top arrow $\kappa_*$ is induced by the homomorphism $\kappa: \Stab(y)\to \Stab(x_{r-1}),\; u\mapsto tut^{-1}$ and the bottom arrow $\delta_*$ is induced by the inner automorphism $G_n \to G_n,\; u\mapsto tut^{-1}$ and the map $C_{r-1}(A)\to C_{r-1}(A),\; x\mapsto t\cdot x$.

By \eqref{commute with stabilizer}, we have $\kappa\circ\jmath=\iota$, so $\kappa_*\circ\jmath_* = \iota_*$. By \cite[Proposition III.8.1]{Br}, the homomorphism $\delta_*$ is the identity map on $H_*(G_n; C_{r-1}(A))$. Therefore, it follows from our two commuting diagrams \eqref{commuting diagram 1} and \eqref{commuting diagram 2} that
\begin{multline*}
\alpha(x_{r-1})_* \circ \iota_* = \alpha(x_{r-1})_* \circ \kappa_*\circ\jmath_* = \delta_*\circ \alpha(y)_* \circ \jmath* \\= \delta_* \circ (d_{i-1})_* \circ \alpha(x_r)_* = (d_{i-1})_* \circ \alpha(x_r)_*.
\end{multline*}
\end{proof}

We are now ready to prove Corollary \ref{stability}.

\begin{proof}[{Proof of Corollary \ref{stability}}]
We use induction on $m$. The base case $m=0$ is trivial.

Suppose $m\geqslant 1$. Let $n>2m$. 

Choose any free resolution $\cdots \to F_1 \to F_0 \to \Z$ of $\Z$ over $\Z G_n$. By taking the tensor product over $G_n$ of the two chain complexes $C_*(A)$ and $F_*$, we obtain a first-quadrant double complex $D$ with $D_{r,s} := F_s\otimes_{G_n} C_r(A)$. Let $\Tot_*(D)$ be the total complex of $D$. From Theorem \ref{connectivity} and the spectral sequence associated to the horizontal filtration of $\Tot_*(D)$, we deduce that $H_i(\Tot_*(D))=0$ for each $i\leqslant n-1$.

We now consider the spectral sequence associated to the vertical filtration of $\Tot_*(D)$. Since $H_i(\Tot_*(D))=0$ for each $i\leqslant n-1$, this spectral sequence has: 
\begin{equation} \label{infinity page}
E^{\infty}_{r,s}=0 \quad \mbox{ if } \quad r+s\leqslant n-1.
\end{equation} 
The $E^1$-terms of the spectral sequence are:
\begin{equation*}
E^1_{r,s} = H_s(F_* \otimes_{G_n} C_r(A)) = H_s(G_n; C_r(A)).
\end{equation*}
The differential $d^1 : E^1_{r,s} \to E^1_{r-1,s}$ is the map
\begin{equation*}
d_* = \sum_{i=1}^r (-1)^{i-1}(d_{i-1})_* : H_s(G_n; C_r(A)) \longrightarrow H_s(G_n; C_{r-1}(A)).
\end{equation*}
Recall that $G_{n-r} = \Stab(x_r)$. We shall identify $E^1_{r,s}$ with $H_s(G_{n-r})$ via the isomorphism $\alpha(x_r)_* : H_s(G_{n-r}) \to H_s(G_n; C_r(A))$. Under this identification, we see from Lemma \ref{lemma for quillen argument} that the differential $d^1 : E^1_{r,s} \to E^1_{r-1,s}$ is:
\begin{itemize}
\item
the map $\iota_*: H_s(G_{n-r}) \to H_s(G_{n-r+1})$ if $r$ is odd;
\item
the zero map if $r$ is even.
\end{itemize}
Therefore, row $s$ on the $E^1$-page of the spectral sequence is:
\begin{equation*}
\xymatrix{
H_s(G_n) & H_s(G_{n-1}) \ar[l]_-{\iota_*} & H_s(G_{n-2}) \ar[l]_-{0} & H_s(G_{n-3}) \ar[l]_-{\iota_*} & H_s(G_{n-4}) \ar[l]_-{0} & \cdots \ar[l]_-{\iota_*},
}
\end{equation*}
where the leftmost term is in column 0.

Our goal is to show that $\iota_*: H_m(G_{n-1}) \to H_m(G_n)$ is an isomorphism, or equivalently, that the differential $d^1 : E^1_{1,m} \to E^1_{0,m}$ is an isomorphism. Note that $E^2_{1,m}$ is the kernel of $d^1 : E^1_{1,m} \to E^1_{0,m}$, and $E^2_{0,m}$ is the cokernel of $d^1 : E^1_{1,m} \to E^1_{0,m}$. Therefore, we have to show that $E^2_{1,m}=0$ and $E^2_{0,m}=0$. We shall use the following:
\begin{claim} 
If $r+2s\leqslant 2m+1$ and $s<m$, then $E^2_{r,s}=0$.
\end{claim}
\begin{proof}[Proof of Claim]
We have: 
\begin{equation*}
2s \leqslant 2m+1-r < n-r+1.
\end{equation*} 
When $r$ is even, we have the stronger inequality:
\begin{equation*}
2s  \leqslant 2m-r < n-r.
\end{equation*}
Hence, it follows by the induction hypothesis that in row $s$ of the $E^1$-page of the spectral sequence, we have:
\begin{equation*}
\xymatrix{
\cdots & H_s(G_{n-r+1}) \ar[l] & H_s(G_{n-r}) \ar[l]_-{d^1_{r,s}} & H_s(G_{n-r-1}) \ar[l]_-{d^1_{r+1,s}} & \cdots \ar[l],
}
\end{equation*}
where
\begin{itemize}
\item
$d^1_{r,s}$ is an isomorphism and $d^1_{r+1,s}$ is the zero map if $r$ is odd;

\item 
$d^1_{r.s}$ is the zero map and $d^1_{r+1,s}$ is an isomorphism if $r$ is even.
\end{itemize}
Hence, $E^2_{r,s}=0$.

We have proven the Claim.
\end{proof}

The above Claim implies that $E^2_{1,m}=E^\infty_{1,m}$ and $E^2_{0,m}=E^\infty_{0,m}$. Indeed, for $k\geqslant 2$, a differential on the $E^k$-page has target $E^k_{1,m}$ or $E^k_{0,m}$ only if it starts at $E^k_{k+1,m-k+1}$ or $E^k_{k, m-k+1}$ respectively, but the Claim implies that $E^k_{k+1,m-k+1}$ and $E^k_{k, m-k+1}$ are both zero.

Finally, since $n \geqslant 2m+1 \geqslant 3$ and so
\begin{equation*}
m+1\leqslant \frac{n+1}{2} \leqslant n-1,
\end{equation*}
it follows from \eqref{infinity page} that $E^\infty_{1,m}=0$ and $E^\infty_{0,m}=0$, so $E^2_{1,m}=0$ and $E^2_{0,m}=0$.
\end{proof}

\begin{remark}
As the last step in the proof above shows, we only need to use the vanishing of $H_r(C_*(A))$ for $|A|\geqslant 3$ and $r\leqslant \frac{|A|+1}{2}$.
\end{remark}

\subsection*{Acknowledgments}
I am grateful to Nathalie Wahl and the referee for their many helpful suggestions to improve the paper.

\end{document}